\documentclass{article}

\usepackage{amssymb,amsthm,amsmath,url}
\usepackage{hyperref}

\title{Arnold stability and Misio{\l}ek curvature}

\author{Taito Tauchi\thanks{
Institute of Mathematics for Industry, Kyushu University, Nishi-ku, Fukuoka, 819-0395, Japan, E-mail address: tauchi.taito.342@m.kyushu-u.ac.jp}
\and 
Tsuyoshi Yoneda\thanks{
Graduate School of Economics, Hitotsubashi University, 2-1
Naka, Kunitachi, Tokyo 186-8601, Japan
E-mail address: t.yoneda@r.hit-u.ac.jp}}

\date{}

\newtheorem{Lemma}{Lemma}[section]
\newtheorem{Fact}[Lemma]{Fact}
\newtheorem{Corollary}[Lemma]{Corollary}
\newtheorem{Theorem}[Lemma]{Theorem}

\theoremstyle{definition}
\newtheorem{Remark}[Lemma]{Remark}
\newtheorem{Definition}[Lemma]{Definition}

\makeatletter
	
	\@addtoreset{equation}{section}
\makeatother

\begin{document}
	\maketitle
	
\begin{abstract}
Let $M$ be a compact 2-dimensional Riemannian manifold with smooth boundary 
and consider the incompressible Euler equation on $M$.
In the case that $M$ is the straight periodic channel,
	the annulus or the disc with the Euclidean metric,
	it was proved by 
	T. D. Drivas, G. Misio{\l}ek, B. Shi, and the second author
	that all Arnold stable solutions have no conjugate point 
	on the volume-preserving diffeomorphism group ${\mathcal D}_{\mu}^{s}(M)$.
	They also proposed a question which asks whether this is true or not for any $M$.
	In this article, we give a partial positive answer.
	More precisely,
	we show that almost all the Misio{\l}ek curvature of any Arnold stable solution is nonpositive.
	The positivity of the Misio{\l}ek curvature 
	is a sufficient condition for the existence of a conjugate point.
\end{abstract}
{\bf MSC2020;} Primary 35Q35; Secondary 35Q31.

	\section{Introduction}
	Let $(M,g)$ be a compact 2-dimensional Riemannian manifold possibly with smooth boundary $\partial M$
and consider the incompressible Euler equation on $M$:
\begin{eqnarray}
	\frac{\partial u}{\partial t} + \nabla_{u}u &=& - \operatorname{grad}p
	\qquad \text{on } M ,
	\nonumber\\
	\operatorname{div}u &=& 0
	\qquad \text{on } M ,
	\label{Euler}
	\\
	g(u,\nu) &=& 0 
	\qquad \text{on }\partial M,
	\nonumber
\end{eqnarray}
where $\nu$ is a unit normal vector field on $\partial M$.
For the case that $M$ is the straight periodic channel,
	the annulus or the disc with the Euclidean metric,
	it was proved by 
	T. D. Drivas, G. Misio{\l}ek, B. Shi, and the second author \cite[Thm.~3]{DMSY}
	that all Arnold stable solutions (see Defition \ref{Def-Arnold-stable}) have no conjugate point on the
	group ${\mathcal D}_{\mu}^{s}(M)$ of volume-preserving Sobolev $H^{s}$ diffeomorphisms 
	   on $M$.
	They also proposed a question \cite[Question 2]{DMSY} which asks whether this is true or not for any 
	compact two-dimensional Riemannian manifold $M$ with smooth boundary.
	In this article,
	we give a partial positive answer.
	For the precise statement,
	we recall the Misio{\l}ek curvature.
	Let $\mu$ be the volume form on $M$ and 
	set
\begin{eqnarray}
	\langle V,W\rangle &:=&
	\int_{M}g(V,W)\mu,
	\label{eq-def-L2-bilinear}
	\\
	|V|^{2}
	&:=&
	\langle V,V\rangle
\end{eqnarray}
for any vector fields $V,W$ on $M$,
which are tangent to $\partial M$.
	\begin{Definition}[{cf.~\cite[(1.3)]{TTconj}, \cite[Lems.~B.6, B.7]{TTCoriolis}}]
	\label{Def-Miciolek-curvature}
		Let $u$ be a stationary solution of \eqref{Euler}
		and $Y$ a divergence-free vector field on $M$,
		which is tangent to $\partial M$.
		The Misio{\l}ek curvature is defined by
		\begin{eqnarray}
			MC_{u,Y}
			:=
			-
			|[u,Y]|^{2}
			-
			\langle
			[[u,Y],Y],
			u
			\rangle.
			\label{eq-Def-Misiolek-curvature}
		\end{eqnarray}
	\end{Definition}
The importance of the Misio{\l}ek curvature 
is the following.
We write $T_{e}{\mathcal D}^{s}_{\mu}(M)$
for the tangent space of ${\mathcal D}^{s}_{\mu}(M)$ at 
the identity element 
$e\in{\mathcal D}^{s}_{\mu}(M)$.
We identify
$T_{e}{\mathcal D}^{s}_{\mu}(M)$
with the space of all Sobolev $H^{s}$ divergence-free vector fields on $M$,
which are tangent to $\partial M$.

\begin{Fact}[{\cite{Misiolek-T} (see also \cite{TTconj})}]
	\label{Misiolek-criterion}
	Let $s>2+\frac{n}{2}$ and $M$ be a compact $n$-dimensional Riemannian 
manifold, possibly with smooth boundary. Suppose that $V\in  T_{e}{\mathcal D}^{s}_{\mu}(M)$ is a stationary solution of the Euler equation \eqref{Euler} on $M$ and take a geodesic $\eta$ on 
${\mathcal D}^{s}_{\mu}(M)$ satisfying $V=\dot{\eta}\circ\eta^{-1}$. Then if $W \in T_{e}{\mathcal D}^{s}_{\mu}(M)$ satisfies $MC_{V,W} > 0$ there
exists a point conjugate to $e\in {\mathcal D}^{s}_{\mu}(M)$ along $\eta(t)$ on $0\leq t\leq t_{0}$ for some $t_{0} >0$.
\end{Fact}

\begin{Remark}
	This was only proved for the case that $M$ has no boundary in \cite{Misiolek-T} (and \cite{TTconj}).
Thus, 
we explain how to apply the proof in \cite{Misiolek-T}
to the case $M$ has a boundary
in the appendix.
\end{Remark}

This fact states that the positivity of the Misio{\l}ek curvature 
ensures the existence of a conjugate point.
This criteria for the existence of a conjugate point 
by using $MC$
was first used in \cite{Misiolek-T} by G. Misio{\l}ek
and recently attracts attention again \cite{DMSY,TTconj,TTCoriolis}.
We note that 
this is only a sufficient condition.
In fact,
there is a stationary solution having a conjugate point,
whose Misio{\l}ek curvature is all nonpositive
(see \cite[Rem.~3]{TTconj}).
However,
philosophically,
the nonpositivity of the Misio{\l}ek curvature
suggests the nonexistence of a conjugate point.
	
Our main theorem of this article is the following.
See Section \ref{Section-Arnold-stable-flow}
for unexplained notions.
\begin{Theorem}
\label{Thm-arnold-stable-stream-nonpositive}
Let $M$ be a two-dimensional Riemannian manifold
possibly with smooth boundary,
$u$ be an Arnold stable solution of \eqref{Euler},
$Y$ a divergence-free vector field on $M$,
which is tangent to $\partial M$.
Suppose that there exist stream functions of $u$ and $Y$. 
Then,
we have
\begin{eqnarray*}
	MC_{u,Y}\leq 0.
\end{eqnarray*}
\end{Theorem}

As a corollary,
we have the following.
Let $S^{1}$ be the one-dimensional sphere
and
$I:=[-1,1]$.
\begin{Theorem}
	\label{Thm-coh-MC-nonpositive}
Let $M$ be a two-dimensional Riemannian manifold
possibly with smooth boundary.
Suppose that either 
$H^{1}_{dR}(M)=0$
or
$M$ is diffeomorphic to $I\times S^{1}$.
Then,
for any Arnold stable solution $u$ of \eqref{Euler}
and 
any divergence-free vector field $Y$ on $M$,
which is tangent to $\partial M$,
we have
\begin{eqnarray*}
	MC_{u,Y}\leq 0.
\end{eqnarray*}
\end{Theorem}

\begin{Remark}
	Note that 
	if $M$ is the disc,
	then we have $H^{1}_{dR}(M)=0$.
	Moreover,
	if $M$ is either the straight periodic channel or the annulus, then $M$ is diffeomorphic to $I\times S^{1}$.
\end{Remark}

\begin{Remark}
It looks like that
	Theorem \ref{Thm-coh-MC-nonpositive} 
	agrees with
	the intuitive argument in \cite{DMSY} before Question 2.
\end{Remark}

\begin{Remark}
Let $a>1$ and
	\begin{eqnarray*}
		M_{a}:=\{(x,y,z)\in{\mathbb R}^{3}\mid x^{2}+y^{2}=a^{2}(1-z^{2})\}
	\end{eqnarray*}
	be a two-dimensional ellipsoid with the Riemannian metric induced by that of ${\mathbb R}^{3}$.
Note that
we have $H^{1}_{dR}(M_{a})=0$
because $M_{a}$ is diffeomorphic to $S^{2}$ for any $a>1$.
Thus, 
Theorem \ref{Thm-coh-MC-nonpositive}
implies 
$MC_{u,Y}\leq 0$ 
for 
any Arnold stable solution $u$ of \eqref{Euler}
and
any divergence-free vector field $Y$ on $M_{a}$.

On the other hand,
Fact \ref{Fact-TTconj-ellipsoid},
which is given below,
implies 
that 
for any zonal flow $u$
(see Definition \ref{Def-zonal-flow} given below for the definition)
 on $M_{a}$ whose support is contained in $M_{a}\backslash \{(0,0,1),(0,0,-1)\}$,
there exists
a divergence-free vector field
$Y$ on $M_{a}$
satisfying $MC_{u,Y}>0$.
This 
implies 
that
any
zonal flow
$u$ on $M_{a}$ whose support is contained in  $M_{a}\backslash \{(0,0,1),(0,0,-1)\}$
never be Arnold stable
by
the assertion of the previous paragraph.
\end{Remark}

\begin{Definition}[{\cite[(1.4)]{TTconj}}]
\label{Def-zonal-flow}
	We say that a vector field $Z$ on $M_{a}$
	is a \textit{zonal flow} if $Z$ has the following form
	\begin{eqnarray*}
		Z=F(z)
		\left(y\frac{\partial}{\partial x}-x\frac{\partial}{\partial y}
		\right)
	\end{eqnarray*}
	for some function $F(z):[-1,1]\to{\mathbb R}$.
\end{Definition}	

Note that a zonal flow is always a stationary solution of the incompressible Euler equation \eqref{Euler} on $M_{a}$.

\begin{Fact}[{\cite[Thm.~1.2]{TTconj}}]
\label{Fact-TTconj-ellipsoid}
Let $a>1$.
Then,
for any
zonal flow
$u$ on $M_{a}$ 
whose support is contained in  $M_{a}\backslash \{(0,0,1),(0,0,-1)\}$,
there exists
a divergence-free vector field
$Y$ on $M_{a}$
satisfying $MC_{u,Y}>0$.
\end{Fact}

By V. I. Arnold \cite{Arnold},
geodesics on ${\mathcal D}_{\mu}^{s}(M)$
correspond to solutions of \eqref{Euler}.
Thus, the existence of a conjugate point is related to a Lagrangian stability of a corresponding solution.

This article is organized as follows.
In Section \ref{Section-Arnold-stable-flow},
we recall the definition and properties of Arnold stability.
In Sections \ref{Section-prf-Thm-assn} and \ref{Sect-prf-cho-MC-nonpositive},
we prove Theorems \ref{Thm-arnold-stable-stream-nonpositive}
and
\ref{Thm-coh-MC-nonpositive},
respectively.
In Appendix A,
we explain how to apply the proof in \cite{Misiolek-T}
to the case $M$ has a boundary.
In Appendix \ref{Appendix-basic-result},
we state the basic results,
which are used in the proof of Theorem \ref{Thm-arnold-stable-stream-nonpositive}.

\section*{Acknowledgments}
The authors are very grateful to G. Misio{\l}ek and 
T. D. Drivas for fruitful discussions. The research of TT was partially supported by Grant-in-Aid for JSPS Fellows (20J00101), Japan Society for the Promotion of Science (JSPS). The research of TY was partially supported by Grant-in-Aid for Scientific Research B (17H02860, 18H01136, 18H01135 and 20H01819), Japan Society for the Promotion of Science (JSPS).

\section{Arnold stable flow}
\label{Section-Arnold-stable-flow}
In this section,
we recall that the definition of an Arnold stable flow and its basic property.
Although almost all the materials in this section are well known, we prove some results for the convenience.
Main references are \cite[Sect.~II.4.A]{AK}, \cite{PTD} and \cite[Sect.~5]{DMSY}.
	
Let $(M,g)$ be a compact 2-dimensional Riemannian manifold possibly with smooth boundary $\partial M$
and consider the incompressible Euler equation \eqref{Euler} on $M$.
\begin{Definition}
Let $u$ be a divergence-free vector field on $M$,
which is tangent to $\partial M$.
A function $\psi$ on $M$ is called a stream function of $u$
if 
$\psi$ satisfies
\begin{eqnarray}
	\star\operatorname{grad}\psi = u, 
	\label{star-grad-psi=u}
\end{eqnarray}
where $\star$ is the Hodge star.
We write 
\begin{eqnarray*}
	\Delta:=\operatorname{div}\circ\operatorname{grad}
\end{eqnarray*}
for the Laplace-Beltrami operator.
In the case \eqref{star-grad-psi=u}, we set
\begin{eqnarray}
\omega:= -\operatorname{div}\star\:u=\Delta\psi.
\label{eq-def-omega-div-grad-psi}
\end{eqnarray}
\end{Definition}

\begin{Lemma}
\label{collinear}
Let $u$ be a stationary solution of \eqref{Euler}
on a two-dimensional Riemannian manifold $M$
possibly with smooth boundary $\partial M$.
Suppose that there exists a function $\psi$ on $M$
such that $u=\star \operatorname{grad}\psi$.
Then $\star\operatorname{grad}\psi$ and $\operatorname{grad}\omega$
are orthogonal. 	
In particular,
$\operatorname{grad}\psi$ and $\operatorname{grad}\omega$ are collinear.
\end{Lemma}
\begin{proof}
	Because $u$ is a time independent solution of \eqref{Euler},
	we have
	\begin{eqnarray}
	\nabla_{u}u=-\operatorname{grad}p,
	\qquad
	\operatorname{div}(u)=0.
	\label{eq-Lem-collinear}
	\end{eqnarray}
Recall that 
$\operatorname{div}(\cdot)=\star d \star (\cdot)^{\flat}
$,
where 
$d$ is the exterior derivative
and
$\flat$ is the musical isomorphism.
We note
that 
the Hodge star
$\star$ commutes with $\flat$
and $\star^{2}=-1$ as an operator on the space of vector fields.
Thus,
applying the operator $\star \circ\operatorname{div}\circ \:\star
= d(\cdot)^{\flat}$ to
the first equation of \eqref{eq-Lem-collinear},
we have
\begin{eqnarray}
 d(\nabla_{u}u)^{\flat}=0.
 \label{eq-Lemma-collinear-dnablauu}
\end{eqnarray}
by
$(\operatorname{grad}p)^{\flat}=dp$ and $d^{2}=0$.
Recall   (cf.~\cite[Thm.~1.17 in Sect.~IV.1.D]{AK})
$$(\nabla_{u}u)^{\flat}
		=
		L_{u}(u^{\flat})
		-\frac{1}{2}d(g(u,u)),
		$$
where $L_{u}$ is the Lie derivative.
Thus, \eqref{eq-Lemma-collinear-dnablauu} implies
\begin{eqnarray}
	L_{u}(d(u^{\flat}))=0.
	\label{eq-Lemma-collinear-L=0}
\end{eqnarray}
by $[L_{u},d]=0$.
On the other hand,
the assumption $u=\star\operatorname{grad}\psi$
implies
\begin{eqnarray*}
	d(u^{\flat})
	&=&
	d(\star(\operatorname{grad}\psi)^{\flat})
	\\
	&=&
	\operatorname{div}(\operatorname{grad}\psi)\mu
	\\
	&=&\omega \mu
\end{eqnarray*}
by $\star \mu =1$ and \eqref{eq-def-omega-div-grad-psi}.
Thus,
\eqref{eq-Lemma-collinear-L=0} implies
\begin{eqnarray*}
	0&=&L_{u}(d(u^{\flat}))
	\\
	&=&
	L_{u}(\omega\mu).
\end{eqnarray*}
By $L_{u}(\mu)=\operatorname{div}(u)\mu=0$ and the Leibniz rule of $L_{u}$,
this is equal to
\begin{eqnarray*}
&=&L_{u}(\omega)\mu
\\
	&=&
	g(u,\operatorname{grad}\omega)\mu,
\end{eqnarray*}
which completes the proof
by $u=\star\operatorname{grad}\psi$.
\end{proof}
	
\begin{Lemma}
\label{F}
Let 
$M$ be
a two-dimensional Riemannian manifold 
possibly with smooth boundary $\partial M$
and
$u$ a stationary solution of \eqref{Euler} on $M$
having
$\psi$ as its stream function.
Set $\omega:=\Delta \psi$.
	Then, 
	there exits a (possibly multivalued) function $F$ on ${\mathbb R}$ satisfying
	$$
	\omega(x) = F(\psi(x))
	\quad
	\text{
	for
	}
	x\in M.
	$$
\end{Lemma}	
\begin{proof}
By Lemma \ref{collinear},
$\operatorname{grad}\psi$ and $\operatorname{grad}\omega$ are collinear.
Thus,
there exits a (possibly multivalued) function $f$
on ${\mathbb R}$
satisfying 
\begin{eqnarray*}
	\operatorname{grad}\omega (x)
	=
	f(\psi(x))
	\operatorname{grad}\psi(x)
	\quad
	\text{
	for
	}
	x\in M.
\end{eqnarray*}	
Take a primitive function $F$ of $f$
(as a function on ${\mathbb R}$).
By the chain rule, we have
\begin{eqnarray}
\operatorname{grad}(F(\psi))=
F'(\psi)\operatorname{grad}\psi
=f(\psi)\operatorname{grad}\psi=\operatorname{grad}\omega.
\label{F'gg}
\end{eqnarray}
Note that
the difference of functions which have the same gradient
must be a constant function.
Thus, adding a suitable constant to $F$
(as a function on ${\mathbb R}$)
if necessary,
we have the lemma.
\end{proof}
\begin{Corollary}
\label{F'}
Let 
$M$ be a two-dimensional Riemannian manifold 
possibly with smooth boundary $\partial M$
and
$u$ be a stationary solution of \eqref{Euler} on $M$
having
$\psi$ as its stream function.
Set
$\omega:=\Delta \psi$.
Then,
the function $F$ in Lemma \ref{F}
satisfies
\begin{eqnarray}
F'(\psi) =\frac{\operatorname{grad} \omega}{\operatorname{grad}\psi}
=
\frac{\operatorname{grad} \Delta \psi}{\operatorname{grad}\psi}.
\label{F'gggg}
\end{eqnarray}
\end{Corollary}
\begin{proof}
This is a consequence of \eqref{F'gg}.
Note that by the collinearity of $\operatorname{grad}\omega$
and $\operatorname{grad}\psi$ (see Lemma \ref{collinear}),
the fraction of \eqref{F'gggg} makes sense.
\end{proof}

Write $\lambda_{1}>0$ for the first eigenvalue of $-\Delta$.
Therefore,
we have
\begin{eqnarray}
	\Delta f\leq -\lambda_{1}f
	\label{eq-def-lambda-laplacian}
\end{eqnarray}
for any function $f$ on $M$
satisfying $\int_{M}f\mu =0$ 
(resp.~$f|_{\partial M}=0$)
if $\partial M$ is empty
(resp.~nonempty),
where $\mu$ is the volume form on $M$.

\begin{Definition}
\label{Def-Arnold-stable}
Let 
$M$ be
a two-dimensional Riemannian manifold 
possibly with smooth boundary $\partial M$.
	We say that a stationary solution $u$ of \eqref{Euler}
is \textit{Arnold stable} if
	the corresponding function $F$ in Lemma \ref{F}
	satisfies
	\begin{eqnarray}
		-\lambda_{1}<F'(\psi)<0,\qquad
		\text{or}
		\qquad
		0<F'(\psi)<\infty.
		\label{eq-Def-Arnold-stable}
	\end{eqnarray}
	\end{Definition}

\begin{Lemma}[{\cite[Prop.\:1.1]{PTD}}]
\label{Arnold-Killing}
	Let 
	$M$ be
	a two-dimensional Riemannian manifold possibly with smooth boundary $\partial M$
	and
	$u$ an Arnold stable stationary solution of \eqref{Euler} with stream function $\psi$.
	Suppose that
	there exits a
		Killing vector field $X$ on $M$,	
	which is tangent to $\partial M$.
	Then we have
	$X\psi=0$.
\end{Lemma}

\begin{proof}
	Note that $\Delta L_{X}=L_{X}\Delta$
	as an operator on the space of functions
	because $X$ is Killing,
	where $L_{X}$ is the Lie derivative.
	By the definition (see \eqref{eq-def-omega-div-grad-psi} and Lemma \ref{F}),
	we have
	$$
	\Delta \psi = F(\psi).
	$$
	The chain rule and $L_{X}\Delta=\Delta L_{X}$ imply
	$$
	(\Delta - F'(\psi))X\psi=0.
	$$
	Thus \eqref{eq-def-lambda-laplacian} and \eqref{eq-Def-Arnold-stable} imply the lemma in the case $\partial M\neq\emptyset$
	because $X\psi|_{\partial M}=0$
	by the assumption that $\psi$ is the stream function of $u$.
	In the case $\partial M=\emptyset$,
	we note that 
	$\int_{M}X\psi\mu = \int_{M}L_{X}(\psi)\mu=0$
	by $L_{X}(\mu)=\operatorname{div}(X)\mu=0$,
	the Leibniz rule of the Lie derivative,
	and the Stokes thoerem.
	Thus,
	\eqref{eq-def-lambda-laplacian} and \eqref{eq-Def-Arnold-stable} also imply the lemma in this case.
\end{proof}

\begin{Remark}
	The equation
	$\Delta L_{X}=L_{X}\Delta$
	is also true as an operator on the space of $p$-forms
	if we interpret that $\Delta$ is the Laplace-de Rham operator
	$\Delta:=(-1)^{n(p+1)+1}
	(d\star d\star +\star d\star d)$,
	where $n:=\dim M$.
	This is because $L_{X}$ commutes the Hodge star operator
	if $X$ is Killing
	(see \cite[(14)]{Deform-Hodge}, for example).
\end{Remark}

\section{Proof of Theorem \ref{Thm-arnold-stable-stream-nonpositive}}
\label{Section-prf-Thm-assn}
In this section,
we prove Theorem \ref{Thm-arnold-stable-stream-nonpositive}.
In the proof,
we use freely 
lemmas in Appendix \ref{Appendix-basic-result}.

\begin{proof}[Proof of Theorem \ref{Thm-arnold-stable-stream-nonpositive}]
By Lemma \ref{Lem-2-dim-Riemannian-is-Kahler},
$(M,g,\omega,\star)$
is an almost K\"aheler manifold,
where $\star$ is the Hodge star operator.
We write $H_{f}$ for the Hamiltonian vector field
of a function $f$ on $M$
(Definition \ref{Def-Hamiltonian-vect}).
By the assumption,
there exist functions $\psi$ and $\phi$
satisfying
\begin{eqnarray*}
	u=\star\operatorname{grad}\psi,
	\quad
	Y=\star\operatorname{grad}\phi
	\qquad
	\in {\mathfrak X}^{t}(M),
\end{eqnarray*}
	where
${\mathfrak X}^{t}(M)$
	is the space of vector fields on $M$,
	which are tangent to ${\partial M}$.
Then,
Lemma \ref{Lem-Hamiltonian-vect-Jgrad-Kahler} 
implies
	\begin{eqnarray*}
		u
		=
		H_{\psi},
		\quad
		Y
		=
		H_{\phi}
		\qquad
	\in {\mathfrak X}^{t}(M).
	\end{eqnarray*}
	Thus,
	we have
	\begin{eqnarray}
		|[u,Y]|^{2}
		&=&
		\langle
		[H_{\psi},H_{\phi}],[H_{\psi},H_{\phi}]
		\rangle
		\nonumber\\
		&=&
		\langle
		H_{\{\psi,\phi\}},H_{\{\psi,\phi\}}
		\rangle
		\nonumber\\
		&=&
		-
		\int_{M}
		\{\psi,\phi\}
		\Delta\{\psi,\phi\}
		\mu
		\label{eq-prf-Thm-MC-1st-term}
	\end{eqnarray}
by Lemmas \ref{Lem-Hamiltonian-vet-Lie-hom} and \ref{Lem-inner-product-on-2D},
where 
$\langle,\rangle$ is given by \eqref{eq-def-L2-bilinear}
and
$\{,\}$ is the Poisson bracket.
	
On the other hand,
we have
\begin{eqnarray*}
	\langle
			[[u,Y],Y],
			u
			\rangle
	&=&
	\langle
	[[H_{\psi},H_{\phi}],H_{\phi}]
	,
	H_{\psi}
	\rangle
	\\
	&=&
	\langle
	H_{\{\{\psi,\phi\},\phi\}}
	,
	H_{\psi}
	\rangle
	\\
		&=&
		\int_{M}
		-
		\{\{\psi,\phi\},\phi\}
		\Delta\psi
		\mu
\end{eqnarray*}
by Lemmas 
\ref{Lem-Hamiltonian-vet-Lie-hom}
and
\ref{Lem-inner-product-on-2D}.
By Lemmas \ref{Lem-int-Poisson-derivative}
and \ref{Lem-tangent-Poisson-zero-on-boundary},
this is equal to
\begin{eqnarray}
		&=&
		-
		\int_{M}
		\{\psi,\phi\}
		\{\phi,
		\Delta\psi\}
		\nonumber
		\\
		&=&
		\int_{M}
		\{\psi,\phi\}
		\{
		\Delta\psi,
		\phi
		\}
		\mu
		\label{eq-prf-Thm-MC-2nodterm}
	\end{eqnarray}
by Lemmas \ref{Lem-Poisson-anti-commute}.

The definition \eqref{eq-Def-Misiolek-curvature} of $MC$
and equations \eqref{eq-prf-Thm-MC-1st-term}, \eqref{eq-prf-Thm-MC-2nodterm}
imply
\begin{eqnarray}
	MC_{u,Y}
	&=&
	\int_{M}
	\{\psi,\phi\}
	\left(
	\Delta\{\psi,\phi\}
	-
	\{\Delta\psi,\phi\}
	\right)
	\nonumber\\
	&=&
	\int_{M}
	H_{\psi}(\phi)
	\left(
	\Delta
	H_{\psi}
	-
	H_{\omega}
	\right)
	(\phi)
	\mu
	\label{eq-prf-Thm-Fgg=FHH}
\end{eqnarray}
by Lemma \ref{Lem-Poisson-formulae}
and \eqref{eq-def-omega-div-grad-psi}.
On the other hand,
there exists a function $F$
satisfying
$$
	F'(\psi)\operatorname{grad}\psi
	=
	\operatorname{grad} \omega
	$$
by
the Arnold stable assumption
(Lemma \ref{F}).
Applying the Hodge star,
we have
\begin{eqnarray}
	F'(\psi)
	H_{\psi}
	=
	H_{\omega}
	\label{eq-prf-Thm-F'HH}
\end{eqnarray}
by Lemma \ref{Lem-Hamiltonian-vect-Jgrad-Kahler}.
Thus,
\eqref{eq-prf-Thm-Fgg=FHH} and \eqref{eq-prf-Thm-F'HH} imply
\begin{eqnarray*}
	MC_{u,Y}
	&=&
	\int_{M}
	H_{\psi}(\phi)
	\left(
	\Delta
	H_{\psi}
	-
	F'(\psi)
	H_{\psi}
	\right)
	(\phi)
	\mu
	\\
	&=&
	\int_{M}
	H_{\psi}(\phi)
	\left(
	\Delta
	-
	F'(\psi)
	\right)
	H_{\psi}
	(\phi)
	\mu.
\end{eqnarray*}
Note that $H_{\psi}(\phi)|_{\partial M}=\{\psi,\phi\}|_{\partial M}=0$
by Lemma \ref{Lem-tangent-Poisson-zero-on-boundary}.
Therefore,
the theorem
is a consequence of \eqref{eq-def-lambda-laplacian}
and \eqref{eq-Def-Arnold-stable}
in the case $\partial M\neq \emptyset$.
Moreover,
if $\partial M = \emptyset$,
we have
\begin{eqnarray*}
	\int_{M}
	H_{\psi}(\phi)\mu 
	&=&
	\int_{M}
	L_{H_{\psi}}(\phi \mu)
	\\
	&=&
	\int_{M}
	d(\iota_{H_{\psi}}(\phi\mu))
	\\
	&=&
	0
\end{eqnarray*}
by $\operatorname{div}(H_{\mu})=0$ (Lemma \ref{Lem-Hamiltonian-vet-div-free})
and the Stokes theorem.
Thus,
\eqref{eq-def-lambda-laplacian}
and \eqref{eq-Def-Arnold-stable}
also imply 
the theorem
in this case.
\end{proof}

\section{Proof of Theorem \ref{Thm-coh-MC-nonpositive}}
\label{Sect-prf-cho-MC-nonpositive}
In this section,
we prove Theorem \ref{Thm-coh-MC-nonpositive}.
Let $M$ be a two-dimensional Riemannian manifold
possibly with smooth boundary $\partial M$.
Recall that ${\mathfrak X}^{t}(M)$
is the space of vector fields on $M$,
which are tangent to $\partial M$.
For the notational simplicity,
we set
\begin{eqnarray*}
{\mathfrak X}_{\mu}^{t}(M)
&:=&
\{
Y\in{\mathfrak X}^{t}(M)\mid
\operatorname{div}(Y)=0
\},
\\
{\mathfrak X}_{\mu}^{t}(M)^{str}&:=&\{Y\in {\mathfrak X}^{t}_{\mu}(M)\mid
Y
\text{
has a stream function}
\},
\\
{\mathfrak X}_{\mu}^{t}(M)^{no}
&:=&
{\mathfrak X}_{\mu}^{t}(M)/
{\mathfrak X}_{\mu}^{t}(M)^{str}.
\end{eqnarray*}
Moreover,
we write
\begin{eqnarray}
	H^{1}_{dR}(M)
	:=
	\{\alpha\in {\mathcal E}^{1}(M)\mid d\alpha=0\}/d(C^{\infty}(M)).
	\label{eq-def-H1-dR}
\end{eqnarray}
for the 1st de Rham cohomology,
where ${\mathcal E}^{1}(M)$ is the space of one-forms on $M$.
Before proving Theorem \ref{Thm-coh-MC-nonpositive}, we need a lemma.

\begin{Lemma}
\label{Lem-no-stream-isom-1st-dR}
	Let $M$ be a two-dimensional Riemannian manifold
	possibly with smooth boundary $\partial M$
	and
	$j:\partial M\hookrightarrow M$ 
	the inclusion.
	Then,
	${\mathfrak X}_{\mu}^{t}(M)^{no}$
	is isomorphic to
	the kernel of $j^{*}:H^{1}_{dR}(M)\to H^{1}_{dR}(\partial M)$,
	where $j^{*}$ is the pull back.
	(We set $H^{1}_{dR}(\partial M):=0$ if $\partial M=\emptyset$.)
\end{Lemma}

\begin{Remark}
	The kernel $j^{*}:H^{1}_{dR}(M)\to H^{1}_{dR}(\partial M)$
	is isomorphic to 
	the relative de Rham cohomology
	$H^{1}(j)$,
	see \cite[Sect.~6 of Ch.~1]{BottTu} or
	\cite[Sect.~8.2]{DF-Weintraub},  for example.
\end{Remark}

\begin{proof}[Proof of Lemma \ref{Lem-no-stream-isom-1st-dR}]
Let $Y$ be a vector field on $M$
(which is not necessarily tangent to $\partial M$).
Note that
	\begin{eqnarray*}
		\operatorname{div}(Y)
		&=&
		\star d (\star Y^{\flat}).
	\end{eqnarray*}
	Thus,
	$Y$ is divergence-free if and only if 
	the one-form $\star Y^{\flat}$
	is closed.
	Therefore,
	we have
	\begin{eqnarray}
		{\mathfrak X}_{\mu}(M) &\simeq & 
		\{\alpha\in {\mathcal E}^{1}(M)\mid d\alpha=0\}
		\label{eq-Lem-isom-div-closed}
		\\
		Y&\mapsto &\star Y^{\flat}, 
		\nonumber
	\end{eqnarray}
	where ${\mathfrak X}_{\mu}(M)$ is the space of divergence-free vector fields
	(which are not necessarily tangent to $\partial M$).
Moreover,
	by definition,
	$Y$ has a stream function 
	if and only if
	\begin{eqnarray*}
		Y=\star\operatorname{grad}\phi
	\end{eqnarray*}
	for some function $\phi$ on $M$.
	Applying the musical isomorphism $\flat$ and 
	the Hodge operator $\star$,
	we have
	\begin{eqnarray*}
		\star Y^{\flat} = -d\phi.
	\end{eqnarray*}
	Thus,
	$Y$ has a stream function
	if and only if 
	the one-form $\star Y^{\flat}$ is exact.
	Therefore,
	we have
	an isomorphism
	\begin{eqnarray}
		\{Y\in{\mathfrak X}_{\mu}(M)\mid
		Y
		\text{
		has a stream function}
		\}
		&\simeq &
		d(C^{\infty}(M))
		\label{eq-Lem-isom-stream-exact}
		\\
		Y
		&\mapsto& \star Y^{\flat}.
		\nonumber
	\end{eqnarray}
	Moreover,
	$Y$ is tangent to $\partial M$
	if and only if 
	\begin{eqnarray*}
		g(\star Y, W)|_{\partial M}=0
	\end{eqnarray*}
	for any vector fields $W$ on $\partial M$
	because 
	$\star$ is the $\frac{\pi}{2}$ rotation operator.
	This equation is equivalent to
	\begin{eqnarray*}
		\star Y^{\flat}(W)|_{\partial M} = 0
	\end{eqnarray*}
	for any vector fields $W$ on $\partial M$.
	Thus,
	we have an isomorphism
	\begin{eqnarray}
		{\mathfrak X}^{t}(M)
		&\simeq &
		\{\alpha\in {\mathcal E}^{1}(M)\mid 
		j^{*}(\alpha)=0\} 
		\label{eq-Lem-isom-tangent-pullback}
		\\
		Y
		&\mapsto& \star Y^{\flat}.
		\nonumber
	\end{eqnarray}
	Then,
	the lemma is a consequence of 
	\eqref{eq-Lem-isom-div-closed},
	\eqref{eq-Lem-isom-stream-exact},
	and \eqref{eq-Lem-isom-tangent-pullback}
	by the definition \eqref{eq-def-H1-dR} of $H^{1}_{dR}(M)$.
\end{proof}

We prove Theorem \ref{Thm-coh-MC-nonpositive}
by using this lemma.

\begin{proof}[Proof of Theorem \ref{Thm-coh-MC-nonpositive}]
	By Theorem \ref{Thm-arnold-stable-stream-nonpositive},
	it is enough to show ${\mathfrak X}^{t}_{\mu}(M)^{no}=0$.
	Moreover,
	by Lemma \ref{Lem-no-stream-isom-1st-dR},
	it is enough to show 
	$j^{*}:H^{1}_{dR}(M)\to H^{1}_{dR}(\partial M)$
	is injective.
	In the case $H^{1}_{dR}(M)=0$,
	this is obvious.
	Therefore, we only consider the case that
	$M$ is diffeomorphic to $I \times S^{1}$.
	Then,
	the de Rham cohomology only depends on the differentiable structure of $M$,
	it is enough to prove the theorem 
	in the case $M= I \times S^{1}$.
	Thus,
	we have to show that 
	if $\alpha \in {\mathcal E}^{1}(I \times S^{1})$
	satisfy $d\alpha =0$ and $j^{*}\alpha =0$,
	then, there exists a function $\phi$ on $I \times S^{1}$
	such that 
	$d\phi =\alpha$.
	For this end,
	we take a coordinate
	$(r,\theta)\in I \times S^{1}$
	and
	$\alpha \in {\mathcal E}^{1}(I \times S^{1})$
	satisfying $d\alpha =0$ and $j^{*}\alpha =0$.
	Write
	\begin{eqnarray}
		\alpha
		=
		f(r,\theta)
		dr
		+
		h(r,\theta)d\theta.
		\label{eq-Lem-def-alpha}
	\end{eqnarray}
	Then,
	$d\alpha=0$ implies
	\begin{eqnarray*}
		(-\partial_{\theta}f+\partial_{r}h)dr\wedge d\theta =0.
	\end{eqnarray*}
	Thus,
	by considering
	the Fourier series
	\begin{eqnarray*}
		f(r,\theta)
		=
		\sum_{n\in{\mathbb Z}} f_{n}(r)e^{in\theta},\qquad
		h(r,\theta)
		=
		\sum_{n\in{\mathbb Z}} h_{n}(r)e^{in\theta},
	\end{eqnarray*}
	we have
	\begin{eqnarray}
		inf_{n} (r)&=&\partial_{r}h_{n}(r)
		\label{eq-prf-Thm-inf=rh}
	\end{eqnarray}
	for all $n\in {\mathbb Z}$.
	In particular,
	we have
	\begin{eqnarray}
		\partial_{r}h_{0}(r)=0.
		\label{eq-prf-Thm-IS1-delr-h0=0}
	\end{eqnarray}
	
	On the other hand,
	$j^{*}(\alpha)=0$ implies
	\begin{eqnarray*}
		h(\pm 1, \theta)
		=
		\sum_{n\in{\mathbb Z}} h_{n}(\pm 1)e^{in\theta}
		=
		 0
	\end{eqnarray*}
	for any $\theta\in S^{1}$
	because
	$j$ is the inclusion
	$\partial (I\times S^{1})=\{\pm 1\}\times S^{1}
	\hookrightarrow I\times S^{1}$.
	In particular,
	we have
	\begin{eqnarray}
		h_{0}(\pm 1)=0.
		\label{eq-prf-Thm-IS1-h0=0}
	\end{eqnarray}
	Thus,
	\eqref{eq-prf-Thm-IS1-delr-h0=0}
	and
	\eqref{eq-prf-Thm-IS1-h0=0}
	imply
	\begin{eqnarray}
		h_{0}=0.
		\label{eq-prf-Thm-h0=0}
	\end{eqnarray}
	
	Take a primitive function $F_{0}(r)$ of $f_{0}(r)$
	and
    define a function $\phi$ on $I\times M$
	by
	\begin{eqnarray*}
		\phi(r,\theta)
		:=
		F_{0}(r)
		+
		\sum_{\substack{n\in {\mathbb Z}\\n\neq0}}
		\frac{h_{n}(r)}{in}e^{in\theta}.
	\end{eqnarray*}
	Then,
	\eqref{eq-Lem-def-alpha}
	and
	\eqref{eq-prf-Thm-inf=rh}
	imply
	\begin{eqnarray*}
		d\phi = \alpha.
	\end{eqnarray*}
	This completes the proof.
\end{proof}

\section*{Appendix: A sufficient criterion of Misio{\l}ek}
\renewcommand{\thesection}{\Alph{section}}
\setcounter{section}{1}
\setcounter{equation}{0}
\label{Misiolek-criterion-Appx}
In this appendix,
we explain how to apply the proof of Fact \ref{Misiolek-criterion} in \cite{Misiolek-T}
to the case $M$ has a boundary.

\subsection{${\mathcal D}_{\mu}^{s}(M)$ in the case $M$ has a boundary}
In this subsection,
we recall briefly the theory of volume-preserving diffeomorphism group ${\mathcal D}_{\mu}^{s}(M)$
in the case that $M$ has a boundary.
Main reference is \cite{EMa}.

Let $M$ be a compact $n$-dimensional Riemannian manifold 
with smooth boundary,
${\mathcal D}_{\mu}^{s}(M)$
the group of all diffeomorphisms of Sobolev
class $H^{s}$ preserving the volume form on $M$.
Then,
the tangent space
$T_{e}{\mathcal D}_{\mu}^{s}(M)$
of ${\mathcal D}_{\mu}^{s}(M)$
at the identity element $e\in {\mathcal D}_{\mu}^{s}(M)$
is identified with the space of divergence-free vector fields on $M$
which are tangent to $\partial M$.
If $s>\frac{n}{2}+1$,
${\mathcal D}_{\mu}^{s}(M)$
has an infinite-dimensional Hilbert manifold structure
with the right-invariant $L^{2}$ Riemannian metric given by
\begin{eqnarray*}
	\langle X, Y \rangle :=
	\int_{M} g(X,Y) \mu,
\end{eqnarray*}
where $X,Y\in T_{e}{\mathcal D}_{\mu}^{s}(M)$.

By V. I. Arnold \cite{Arnold},
a solution $u$ of the incompressible Euler equation \eqref{Euler} on $M$ 
corresponds to a geodesic $\eta$ on ${\mathcal D}_{\mu}^{s}(M)$
starting at $e\in{\mathcal D}_{\mu}^{s}(M)$
via $u=\dot{\eta}\circ\eta^{-1}$.
Thus,
it is important to study of the geometry of ${\mathcal D}_{\mu}^{s}(M)$.
In particular, the existence of a conjugate point on a geodesic has attractive considerable attention because
it is related to the Lagrangian stability of the corresponding solution.

\subsection{Sketch of the proof of Fact \ref{Misiolek-criterion}}
\label{Appx-Sketch-Misiolek-criterion}
In this subsection,
we explain how to apply the proof of Fact \ref{Misiolek-criterion} in \cite{Misiolek-T} to the case that $M$ has a boundary.
For the convenience,
we rewrite
Fact \ref{Misiolek-criterion}.

\renewcommand{\thesection}{\arabic{section}}
\setcounter{section}{1}
\setcounter{Lemma}{1}
\begin{Fact}
	Let $M$ be a compact $n$-dimensional Riemannian 
manifold with smooth boundary and $s>2+\frac{n}{2}$. Suppose that $V\in  T_{e}{\mathcal D}^{s}_{\mu}(M)$ is a stationary solution of the Euler equation \eqref{Euler} on $M$ and take a geodesic $\eta$ on 
${\mathcal D}^{s}_{\mu}(M)$ satisfying $V=\dot{\eta}\circ\eta^{-1}$. Then if $W \in T_{e}{\mathcal D}^{s}_{\mu}(M)$ satisfies $MC_{V,W} > 0$ there
exists a point conjugate to $e\in {\mathcal D}^{s}_{\mu}(M)$ along $\eta(t)$ on $0\leq t\leq t_{0}$ for some $t_{0} >0$.
\end{Fact}
\renewcommand{\thesection}{\Alph{section}}
\setcounter{section}{1}

\begin{proof}[Sketch of the proof of Fact \ref{Misiolek-criterion}]
Because the Riemannian metric of ${\mathcal D}_{\mu}^{s}(M)$ is right invariant,
Theorem B.5 in \cite{TTCoriolis}
shows that there exist $t_{0}>0$ and a vector field $\widetilde{W}$ on $\eta$ satisfying
$\widetilde{W}(0)=\widetilde{W}(t_{0})=0$
and
\begin{eqnarray}
	E''(\eta)_{0}^{t_{0}}(\widetilde{W},\widetilde{W})<0
	\label{E''<0}
\end{eqnarray}
by the assumption $MC_{V,W}>0$.
Here 
$E''(\eta)_{0}^{t_{0}}(\widetilde{W},\widetilde{W})$
is the second variation of the energy function $E_{0}^{t_{0}}(\eta)$ of $\eta$:
$$
E_{0}^{t_{0}}(\eta):=\frac{1}{2}\int_{0}^{t_{0}}\langle \dot{\eta},\dot{\eta}\rangle dt.
$$

On the other hand,
the same argument of \cite[Lem.~3]{Misiolek-T}
gives 
\begin{eqnarray}
E''(\eta)_{0}^{t_{0}}(Z,Z)\geq 0
\label{E''>0}
\end{eqnarray}
for any vector field $Z(t)$ on $\eta$ with $Z(0)=Z(t_{0})=0$
if there exists no conjugate point on $\eta(t)$ ($0\leq t\leq t_{0}$).
The essential point of the argument of \cite[Lem.~3]{Misiolek-T}
is that the differential of the exponential map is bounded operator,
which is deduced by
the boundedness of the curvature of ${\mathcal D}_{\mu}^{s}(M)$ in \cite[Lem.~3]{Misiolek-T}.
This boundedness of the curvature is also guaranteed for the case that $M$ has a boundary by \cite[Prop.~3.6]{Misiolek-Stability}.
Thus,
the same argument is valid in the case that $M$ has a boundary
and the contradiction of \eqref{E''<0} to \eqref{E''>0}
gives the desired result.
\end{proof}

\section{Some basic results}
\label{Appendix-basic-result}
In this section,
we recall basic results on symplectic and almost K\"ahler manifolds.
Although almost all the materials in this section are well known, we prove some results for the convenience.
Main references are 
\cite[Sect.~4]{Kahler-ballman},
\cite[Sect.~22]{ISM-Lee}
and
\cite[Sect.~2]{almost-Kahler-ST}.

\subsection{Symplectic manifold with boundary}
Let $(M,\omega)$ be a compact symplectic manifold
possibly with smooth boundary $\partial M$.
We write ${\mathfrak X}(M)$
(resp.~${\mathfrak X}^{t}(M)$)
for the space of vector fields on $M$
(resp.~which are tangent to $\partial M$).

\begin{Definition}
\label{Def-Hamiltonian-vect}
Let $f\in C^{\infty}(M)$.
Then, the Hamilton vector field 
$H_{f}\in{\mathfrak X}(M)$ 
of $f$ 
is defined by 
the equation
\begin{eqnarray}
	\iota_{H_{f}}\omega=df
	\label{eq-Def-Hamiltonian-vect}
\end{eqnarray}
where $d$ is the exterior derivative
and $\iota_{H_{f}}$ is the interior derivative.
\end{Definition}

We always take 
\begin{eqnarray}
\mu:=\frac{1}{n!}\omega^{n}
:=
\frac{1}{n!}
\underbrace{\omega\wedge\dots\wedge\omega}_\text{$n$ times}.
\label{eq-def-volume-form-on-symplectic}
\end{eqnarray}
as the volume form on $M$,
where $n:=\frac{\dim M}{2}$.

\begin{Definition}
	\label{Def-divergence-on-Symplectic}
Let $V\in{\mathfrak X}(M)$.
The divergence of $V$
is defined by
\begin{eqnarray*}
	\operatorname{div}(V)\mu
	=
	L_{V}(\mu)
\end{eqnarray*}
where $L_{V}$ is the Lie derivative.
\end{Definition}

\begin{Lemma}
	\label{Lem-Hamiltonian-vet-div-free}
Let $f\in C^{\infty}(M)$.
Then,
we have
\begin{eqnarray*}
	\operatorname{div}(H_{f})=0.
\end{eqnarray*}
\end{Lemma}

\begin{proof}
By \eqref{eq-def-volume-form-on-symplectic}
and
the Cartan magic formula
$L_{H_{f}}=d\circ\iota_{H_{f}}+\iota_{H_{f}}\circ d$,
we have
	\begin{eqnarray*}
		n!L_{H_{f}}(\mu)
		&=&
		d(
		\iota_{H_{f}}
		(\omega^{n}))
	\end{eqnarray*}
because $d\omega=0$.
By the graded Leibniz rule of the interior derivative
and
\eqref{eq-Def-Hamiltonian-vect},
this is equal to	
	\begin{eqnarray*}
		&=&
		n
		d(df\wedge \omega^{n-1}).
	\end{eqnarray*}	
By the Leibniz rule of $d$,
	and $d^{2}=0$,
	this is equal to
	\begin{eqnarray*}
		&=&
		n\left(
		ddf\wedge \omega^{n-1}
		-df\wedge (d\omega^{n-1})
		\right)
		\\
		&=&
		0,
	\end{eqnarray*}
which completes the proof.
\end{proof}

\begin{Definition}
\label{Def-Poisson-bracket}
Let $f,g\in C^{\infty}(M)$.
The Poisson bracket of $f$ and $g$
is defined by
\begin{eqnarray}
	\{f,g\}
	:=
	-\omega(H_{f},H_{g})
	=
	\omega(H_{g},H_{f}).
	\label{eq-Def-Poisson-bracket}
\end{eqnarray}
	
\end{Definition}

\begin{Lemma}
\label{Lem-Poisson-anti-commute}
For $f,g\in C^{\infty}(M)$,
we have
$$
	\{f,g\}
	=
	-
	\{g,f\}.$$
\end{Lemma}
\begin{proof}
	By the skew-symmetry of $\omega$ and the definition \eqref{eq-Def-Poisson-bracket},
	this lemma is obvious.
\end{proof}

\begin{Lemma}
\label{Lem-Poisson-formulae}
For $f,g\in C^{\infty}(M)$,
we have
$$
	\{f,g\}
	=
	-df(H_{g})
	=
	-H_{g}(f)
	=
	dg(H_{f})
	=
	H_{f}(g).
$$
\end{Lemma}

\begin{proof}
	This is obvious
	from \eqref{eq-Def-Hamiltonian-vect},
	\eqref{Def-Poisson-bracket}
	and the definition of the exterior derivative $d$.
\end{proof}

\begin{Lemma}
\label{Lem-Poisson-derivation}
	For $f,g\in C^{\infty}(M)$,
	we have
	\begin{eqnarray}
		\{f,gh\}=
		\{f,g\}h
		+
		\{f,h\}g.
		\label{eq-Lem-Poisson-derivation}
	\end{eqnarray}
\end{Lemma}

\begin{proof}
	Lemma \ref{Lem-Poisson-formulae}
	implies
	\begin{eqnarray*}
		\{f,gh\}
		&=&
		d(gh)(H_{f})
		\\
		&=&
		hdg(H_{f})
		+
		gdh(H_{f})
	\end{eqnarray*}
	by the Leibniz rule of $d$.
	This completes the proof by Lemma \ref{Lem-Poisson-formulae}.
\end{proof}

\begin{Lemma}
\label{Lem-Hamiltonian-vet-Lie-hom}
For $f,g\in C^{\infty}(M)$,
we have
$$
[H_{f},H_{g}]=H_{\{f,g\}}.
$$
\end{Lemma}

\begin{proof}
	Recall that the Lie derivative  and the interior derivative
	satisfy
	\begin{eqnarray}
		\iota_{[V,W]}
		&=&
		L_{V}\circ \iota_{W}
		-
		\iota_{W}\circ L_{V},
		\\
		L_{V}
		&=&
		\iota_{V}\circ d +d\circ \iota_{V}
		\label{eq-Cartan-magic}
	\end{eqnarray}
	for any $V,W\in {\mathfrak X}(M)$.
	Thus, we have
	\begin{eqnarray*}
		\iota_{[H_{f},H_{g}]}
		(\omega)
		&=&
		(L_{H_{f}}\circ \iota_{H_{g}}
		-
		\iota_{H_{g}}\circ L_{H_{f}})(\omega)
		\\
		&=&
		(\iota_{H_{f}}\circ d
		 +
		 d\circ \iota_{H_{f}})
		 (\iota_{H_{g}}(\omega))
		 -
		 \iota_{H_{g}}\circ 
		L_{H_{f}}
		  (\omega)
	\end{eqnarray*}
Moreover,
we have
\begin{eqnarray*}
	d\iota_{H_{g}}(\omega)
	&=&ddg
	\\
	&=&0,
	\\
	L_{H_{f}}(\omega)
	&=&
	(\iota_{H_{f}}\circ d
		 +
		 d\circ \iota_{H_{f}})(\omega)
		 \\
		 &=&0
\end{eqnarray*}
by $d\omega=0$.
These impliy
\begin{eqnarray*}
	\iota_{[H_{f},H_{g}]}
		(\omega)
		&=&
		d\circ 
		\iota_{H_{f}}
		\circ\iota_{H_{g}}(\omega)
		\\
		&=&
		d(dg(H_{f})).
\end{eqnarray*}
This completes the proof by Definition \ref{Def-Hamiltonian-vect}
and
Lemma \ref{Lem-Poisson-formulae}.
\end{proof}

\subsection{Almost K\"ahler manifold}

Let $(M,g,\omega,J)$
be a almost K\"ahler manifold
possibly with smooth boundary $\partial M$.
Namely,
$g$ is a Riemannian metric on $M$,
$\omega$ is a symplectic form on $M$,
and 
$J$ is an operator on the tangent bundle $TM$ on $M$
satisfying
\begin{eqnarray}
J^{2}&=&-1,
\label{eq-def-Kahler-JJ=-1}
\\
	g(V,W)&=&\omega(JV,W)
	\label{eq-def-Kahler-gw}
\end{eqnarray}
for any $V,W\in {\mathfrak X}(M)$.

\begin{Lemma}
\label{Lem-g-omea-J-relation-Kahler}
Let $V,W\in {\mathfrak X}(M)$.
Then,
we have
\begin{eqnarray*}
g(JV,JW)&=&g(V,W)
\end{eqnarray*}
for any $V,W\in {\mathfrak X}(M)$.
\end{Lemma}

\begin{proof}
By \eqref{eq-def-Kahler-JJ=-1}, \eqref{eq-def-Kahler-gw},
and the skew-symmetry of $\omega$,
we have
\begin{eqnarray*}
g(JV,JW)=-\omega(V,JW)=\omega(JW,V)=g(W,V)=g(V,W).
\end{eqnarray*}
This completes the proof.
\end{proof}

\begin{Lemma}
	\label{Lem-Hamiltonian-vect-Jgrad-Kahler}
	Let $f\in C^{\infty}(M)$.
	Then,
	we have
	\begin{eqnarray*}
		H_{f}
		=
		J\operatorname{grad}f.
	\end{eqnarray*}
\end{Lemma}

\begin{proof}
	By the definition of the gradient,
	we have
	\begin{eqnarray*}
		df(\cdot)
		&=&
		g(\operatorname{grad}f,\cdot).
	\end{eqnarray*}
This implies the lemma
by Definition \ref{Def-Hamiltonian-vect}
and \eqref{eq-def-Kahler-gw}.
\end{proof}

\begin{Lemma}
	\label{Lem-int-Poisson-Stokes}
Let $f,g\in C^{\infty}(M)$.
Then,
we have
\begin{eqnarray*}
	\int_{M}\{f,g\}\mu
	&=&
	-\int_{\partial M}
	f
	\iota_{H_{g}}(\mu).
\end{eqnarray*}
\end{Lemma}

\begin{proof}
	By Lemma \ref{Lem-Poisson-formulae},
	we have
	\begin{eqnarray*}
		\int_{M}
		\{f,g\}
		\mu
		&=&
		-
		\int_{M}
		H_{g}(f)\mu
		\\
		&=&
		-\int_{M}
		L_{H_{g}}(f)\mu.
	\end{eqnarray*}
Note $\operatorname{div}(H_{f})=0$ by Lemma \ref{Lem-Hamiltonian-vet-div-free}.
Thus,
this is equal to
\begin{eqnarray*}
	&=&
	-\int_{M}
	L_{H_{g}}(f\mu)
	\\
	&=&
	-\int_{M}
	d(\iota_{H_{g}}(f\mu))
\end{eqnarray*}
by the Leibniz rule of the Lie derivative
and the Cartan magic formula \eqref{eq-Cartan-magic}.
Thus,
the Stokes theorem implies the lemma.
\end{proof}

\begin{Lemma}
\label{Lem-int-Poisson-derivative}
For any $f,g,h\in C^{\infty}(M)$,
we have
$$
\int_{\partial M}
	fh\iota_{H_{g}}(\mu)
=
\int_{M}
		\left(
		-
		\{f,g\}h
		+
		f
		\{g,h\}
		\right)
		\mu.
$$
In particular,
if $fh|_{\partial M}=0$,
we have
\begin{eqnarray*}
	\int_{M}
		\{f,g\}h
		\mu
		=
		\int_{M}
		f\{g,h\}
		\mu.
\end{eqnarray*}
\end{Lemma}

\begin{proof}
	By Lemma \ref{Lem-Poisson-derivation},
	we have
	\begin{eqnarray*}
		\int_{M}
		\{g,fh\}\mu
		=
		\int_{M}
		\left(
		\{g,f\}h
		+
		f
		\{g,h\}
		\right)
		\mu.
	\end{eqnarray*}
	By Lemmas \ref{Lem-Poisson-anti-commute} and \ref{Lem-int-Poisson-Stokes},
	we have the lemma.
\end{proof}

\begin{Lemma}
\label{Lem-int-HH-gradgrad}
For $f,g\in C^{\infty}(M)$,
we have
$$
\int_{M}g
(H_{f},
H_{g})\mu
=
\int_{M}g
(\operatorname{grad}f,
\operatorname{grad}g)
\mu.
$$
\end{Lemma}

\begin{proof}
	This is obvious by Lemmas \ref{Lem-g-omea-J-relation-Kahler}
	and \ref{Lem-Hamiltonian-vect-Jgrad-Kahler}.
\end{proof}

\subsection{$L^{2}$ inner product on almost K\"ahler manifold}

Let $(M,g,\omega,J)$
be an almost K\"ahler manifold
possibly with smooth boundary $\partial M$.
	Set
\begin{eqnarray}
	\langle V,W\rangle &:=&
	\int_{M}g(V,W)\mu,
	\\
	|V|^{2}
	&:=&
	\langle V,V\rangle
\end{eqnarray}
for any $V,W\in {\mathfrak X}(M)$.

\begin{Definition}
	\label{Def-laplacian-Beltrami}
The Laplace-Beltrami operator is defined by
\begin{eqnarray*}
	\Delta:=\operatorname{div}\circ\operatorname{grad}.
\end{eqnarray*}
\end{Definition}

\begin{Lemma}
	\label{Lem-L2-Kahler-laplacian}
Let $f,g\in C^{\infty}(M)$.
Then, we have
\begin{eqnarray*}
	\langle H_{f},H_{g}\rangle
	=
	\int_{\partial M}
	f\iota_{\operatorname{grad}g}(\mu)
	-
	\int_{M}
	f\Delta(g)
	\mu.
\end{eqnarray*}
In particular,
if $f|_{\partial M}=0$,
we have
\begin{eqnarray*}
	\langle H_{f},H_{g}\rangle
	=
	-
	\int_{M}
	f\Delta(g)\mu
\end{eqnarray*}
\end{Lemma}

\begin{proof}
	We have
	\begin{eqnarray*}
		\langle H_{f},H_{g}\rangle
		&=&
		\int_{M}
		g(H_{f},H_{g})
		\mu
		\\
		&=&
		\int_{M}
		g(\operatorname{grad}f,\operatorname{grad}g)
		\mu\\
		&=&
		\int_{M}
		L_{\operatorname{grad}g}
		(f)
		\mu
	\end{eqnarray*}
	by Lemma \ref{Lem-int-HH-gradgrad} 
	and the definition of the gradient.
	By the Leibniz rule of the Lie derivative,
	this is equal to
	\begin{eqnarray*}
		&=&
		\int_{M}
		L_{
		\operatorname{grad}g}
		(f\mu)
		-
		f
		L_{\operatorname{grad}g}(\mu)
		\\
		&=&
		\int_{M}
		d\circ\iota_{
		\operatorname{grad}g}
		(f\mu)
		-
		f
		\Delta(g)\mu.
	\end{eqnarray*}
	This completes the proof
	by the Stokes theorem.
\end{proof}

\subsection{2D Riemannian manifold}
Let $M$ be an orientable two-dimensional Riemannian manifold
possibly with smooth boundary $\partial M$.
Note that $\dim M=2$ implies
that the Hodge star operator $\star$ satisfies
\begin{eqnarray*}
	\star^{2}=-1
\end{eqnarray*}
as an operator on ${\mathfrak X}(M)$.

\begin{Lemma}
\label{Lem-2-dim-Riemannian-is-Kahler}
	Define a two-form $\omega$ on $M$
	by
	\begin{eqnarray*}
		\omega(V,W):=g(\star V, W),
	\end{eqnarray*}
	where $V,W\in {\mathfrak X}(M)$.
	Then,
	$(M,g,\omega,\star)$
	is an almost K\"ahler manifold.
\end{Lemma}

\begin{proof}
	This follows from the definition.
\end{proof}

\begin{Lemma}
\label{Lem-tangent-Poisson-zero-on-boundary}
	Let $f,g\in C^{\infty}(M)$
	with
	$H_{f},H_{g}\in {\mathfrak X}^{t}(M)$.
	Then,
	we have
	\begin{eqnarray*}
		\{f,g\}|_{\partial M}=0.
	\end{eqnarray*}
\end{Lemma}

\begin{proof}
Note that $H_{f}$ and $H_{g}$ are tangent to $\partial M$ by the assumption.
Therefore,
we have
\begin{eqnarray*}
	g(\star H_{f}, H_{g})|_{\partial M}
		&=&
		0
\end{eqnarray*}
because
$\star$ is the $\frac{\pi}{2}$ rotation operator.
On the other hand,
we have
\begin{eqnarray*}
	\{f,g\}
	&=&
	-\omega(H_{f},H_{g})
	\\
	&=&
	g(\star H_{f}, H_{g})
\end{eqnarray*}
by 
Definition \ref{Def-Poisson-bracket}
and
\eqref{eq-def-Kahler-gw}.
This completes the proof.
\end{proof}

\begin{Lemma}
	\label{Lem-Poisson-triple-zero-on-boundary}
	Let $f,g,h\in C^{\infty}(M)$
	with 
	$H_{f},H_{g},H_{h}\in {\mathfrak X}^{t}(M)$.
	Then,
	we have
	\begin{eqnarray*}
		\{\{f,g\},h\}|_{\partial M}=0.
	\end{eqnarray*}
\end{Lemma}

\begin{proof}
	By Lemma \ref{Lem-tangent-Poisson-zero-on-boundary},
	$\{f,g\}$ is constant on $\partial M$.
	Thus,
	we have the lemma because
	$H_{h}$ is tangent to $\partial M$
	and
	$\{\{f,g\},h\} = -H_{h}(\{f,g\})$
	by Lemma \ref{Lem-Poisson-formulae}.
\end{proof}

\begin{Lemma}
	\label{Lem-inner-product-on-2D}
Let $f,g,h\in C^{\infty}(M)$
	with
	$H_{f},H_{g}\in {\mathfrak X}^{t}(M)$.
	Then,
	we have
\begin{eqnarray*}
	\langle H_{\{f,g\}},H_{h}\rangle
	=
	-
	\int_{M}
	\{f,g\}\Delta(h)\mu,
	\\
	\langle H_{\{\{f,g\},g\}},H_{h}\rangle
	=
	-
	\int_{M}
	\{
	\{f,g\},g\}\Delta(h)\mu.
\end{eqnarray*}
\end{Lemma}

\begin{proof}
	This follows from
	Lemmas \ref{Lem-L2-Kahler-laplacian},
	\ref{Lem-tangent-Poisson-zero-on-boundary},
	and
	\ref{Lem-Poisson-triple-zero-on-boundary}.
\end{proof}

\end{document}